\documentclass[11pt,reqno]{amsart}

\usepackage{amsfonts,amsmath,amsthm,amssymb,mathtools}
\usepackage{enumerate}
\usepackage{bm}

\usepackage{hyperref}

\newtheorem{theorem}{Theorem}[section]
\newtheorem{lemma}[theorem]{Lemma}
\newtheorem{proposition}[theorem]{Proposition}
\newtheorem{corollary}[theorem]{Corollary}

\theoremstyle{definition}

\newtheorem*{ack}{Acknowledgments}

\theoremstyle{remark}
\newtheorem{remark}[theorem]{Remark}

\numberwithin{equation}{section}

\DeclarePairedDelimiter{\parens}{(}{)}
\DeclarePairedDelimiter{\inner}{\langle}{\rangle}

\newcommand*{\op}[1]{\operatorname{#1}}
\newcommand*{\R}{\mathbb{R}}

\numberwithin{equation}{section}

\newcounter{rom}
\renewcommand{\therom}{(\roman{rom})}
{\end{list}}

\title[anisotropic capillary hypersurfaces in a wedge]
    {Anisotropic capillary hypersurfaces in a wedge}
\author{Hui Ma}
\address{Department of Mathematical Sciences, Tsinghua University,
Beijing 100084, P.R. China} 
\email{ma-h@tsinghua.edu.cn}
\author{Jiaxu Ma}
\email{mjx22@mails.tsinghua.edu.cn}
\author{Mingxuan Yang}
\address{Academy of Mathematics and Systems Science, the Chinese Academy of Sciences, Beijing 100190, China}
\email{ymx20@amss.ac.cn}

\date{}

\begin{document}

\begin{abstract}
We investigate anisotropic capillary hypersurfaces within a wedge in Euclidean space.
In this study, we generalize the Minkowski norm \(F\), traditionally employed to define the anisotropic surface energy, to a gauge on the unit sphere \(S^n\).
This generalization helps to illuminate a significant relationship between capillary hypersurfaces and hypersurfaces with free boundary.
Our main results include new Minkowski formulae and a Heintze-Karcher type inequality.
As an application, we prove an Alexandrov-type theorem, thereby extending the known results to the anisotropic setting.
\end{abstract}



\maketitle
  
\section{Introduction}\label{sec:intro}

Capillary surfaces model the shape of incompressible liquid droplets supported on a solid substrate in microgravity conditions.
The contact angle depends on the materials involved (see \cite{ConcusFinn, Finn}).

Modern studies of capillary surfaces can be traced back to the pioneering work of Young, Laplace, and Gauss in the early 19th century.
Their research established the foundation for understanding how immiscible materials interact at interfaces,
aiming to minimize surface energy under given physical conditions. 

For solid-state materials or liquid crystals, even when the surface material appears uniform, the classical isotropic surface energy must be replaced by an anisotropic surface energy to accurately describe the situation.
This shift is critical for understanding numerous phenomena in materials science, physics, and engineering, such as alloy casting and thin film growth.

The investigation of anisotropic capillary hypersurfaces within various domains, such as upper half-spaces, balls, cylinders, cones, and wedges, is of significant importance. Recently, there have been many results in this field, including stability problems, overdetermined problems and the regularity of minimizing capillary hypersurfaces (refer to \cite{CEL, CK, GWX, Koiso, LX, MYY, Rosales, WX1} and the references therein).

In wedge domains, 
liquid droplets are stabilized at certain Wulff shapes \cite{Koiso, Motoki}.
This behavior has practical applications, such as in the distribution of liquid propellants, guiding the liquids to wedge-shaped exit positions even in low-gravity environments, as discussed in \cite{McCuan}.

This paper focuses on the study of anisotropic capillary hypersurfaces in wedge domains, particularly addressing an Alexandrov-type theorem for anisotropic capillary hypersurfaces intersecting with the two hyperplanes forming the wedge domain.

%%%%%

Recall that Alexandrov's Soap Bubble Theorem \cite{Alex1} is a fundamental result in differential geometry, stating that any closed embedded hypersurface in \(\mathbb{R}^{n+1}\) with constant mean curvature must be a round sphere. 
It establishes important links between the mean curvature and the rigidity of a closed hypersurface. 
The various methods to prove the Alexandrov theorem bring distinct perspectives and significant tools. For instance,
Alexandrov's approach involves reflecting through moving planes based on the maximum principle; Reilly \cite{Reilly} presented a different proof via his famous integration formula; Montiel and Ros  \cite{Montiel-Ros} combined the Minkowski formula and the Heintze-Karcher inequality to offer a geometric proof; The proof by Hajazi, Montiel, and Zhang \cite{HMZ} uses a spinorial Reilly-type formula. Brendle \cite{brendle} achieved a significant generalized Alexandrov theorem in warped product manifolds by employing the normal geodesic flow with respect to a conformal metric.

About the Alexandrov theorem on capillary hypersurfaces,
Wente \cite{Wente} initially studied the case of half-space \(\mathbb{R}^{n+1}_{+}\).
Park \cite{Park} later obtained a similar result for the ring-type spanner
in a wedge by Alexandrov’s reflection argument. 
Subsequently, Choe and Park \cite{CP} and López \cite{Lopez} considered the theorem in convex cone and wedge by the method of Reilly formula, respectively. 
Pyo \cite{Pyo} got rigidity theorems of hypersurfaces with free boundary in a wedge in a space form.
Jia, Wang, Xia and Zhang \cite{JWXZ} further obtained the Alexandrov theorem for capillary hypersurfaces in a wedge.
More recently, Wang and Xia \cite{WX} introduced a novel perspective on capillary hypersurfaces in a unit ball by reducing them to free boundary cases via a special Finsler metric, thereby leading to the Alexandrov theorem.

For the anisotropic setting, He, Li, Ma and Ge \cite{HLMG} first proved the anisotropic version of the Alexandrov theorem for closed hypersurfaces in Euclidean space. Recently, Jia, Wang, Xia and Zhang \cite{JWXZ2} extended Wente's result to anisotropic capillary hypersurfaces in the half-space \(\mathbb{R}^{n+1}_{+}\).

Concerning anisotropic capillary hypersurfaces in a classical wedge, we first introduce the setup.
Let \(\bm{n}_1, \bm{n}_2\) be two linearly independent unit vectors in \(\R^{n+1}\).
A wedge $\mathbf{W}$ determined by \(\bm{n}_1,\bm{n}_2\) is defined to be the set
\[
    \{x\in\R^{n+1}\mid\inner{x,\bm{n}_i}<0,\, i=1,2\}.
\]
This is an open region in \(\mathbb{R}^{n+1}\) bounded by two half-hyperplanes.
The closure \(\overline{\mathbf{W}}\) is a smooth manifold with corners.
The boundary \(\partial \mathbf{W}\) consists of two open half-hyperplanes \(P_1,P_2\)
and a codimension $2$ linear subspace \(L\) in \(\mathbb{R}^{n+1}\).
Here \(L\) is the boundary of \(P_1\) and \(P_2\).

Consider \(\Sigma\subset\overline{\mathbf{W}}\) as a smooth, compact hypersurface within a classical wedge, with \(\partial\Sigma\subset\partial\mathbf{W}\), and let \(\Omega\) denote the domain enclosed by \(\Sigma\) and \(\partial \mathbf{W}\).
Following \cite{Koiso}, the energy functional is defined as
\[
    E(\Sigma) = \int_{\Sigma} F(\bm{\nu}) \, \mathrm{d} A + \sum_{i=1}^2\omega_{0}^i \left|\partial \Omega \cap P_{i}\right|,
\]
where \(\bm{\nu}\) is the unit outward normal vector field of \(\Sigma\) and \(F\) is a \textit{gauge}, i.e.
a nonnegative convex function defined on \(\mathbb{R}^{n+1}\) which is positively homogeneous of degree one (see \cite[(1.6)]{CRS}).
The first term \(\int_{\Sigma} F(\bm{\nu}) \,  \mathrm{d} A\) represents the anisotropic surface energy.
The second term $\omega_0^i \left|\partial \Omega \cap P_{i}\right|$ represents the wetting energy, where \(\omega_{0}^i\in\mathbb{R}\) is a given constant.

The gauge \(F\) is the support function of a closed convex body \(K\subset\R^{n+1}\) satisfying \(\bm{0}\in K\).
In some situations, it is useful to consider \(K\) instead of \(F\). It is worth noting that, in previous works on anisotropic hypersurfaces, \(F\) is usually taken to be positive.
We allow \(F\) to \textit{degenerate} on a ray \(\ell\) starting at \(\bm{0}\), i.e. \(F(x)=0\) for \(x\in\ell\).
When focusing on \(K\), the degeneracy of \(F\) is equivalent to \(\bm{0}\in\partial K\).
In the proof of our main theorems, it is necessary to deal with degenerate gauges separately.

Constrained by a fixed volume, we examine the first variation of the energy functional $E$.
For a family of hypersurfaces $\left\{\Sigma_t\right\}$ that vary smoothly,
with boundaries $\partial \Sigma_t$ moving freely on the boundary $\partial \mathbf{W}$,
and according to a variational vector field $Y$
such that $\left.Y\right|_{\partial \Sigma} \in T\left(\partial \mathbf{W}\right)$,
the first variation formula of $E$ is given by 
\[
    \left.\frac{\mathrm{d}}{\mathrm{d}t}\right|_{t=0} E(\Sigma_t)
    =\int_{\Sigma} H_1^F\inner{Y,\bm{\nu}}\,\mathrm{d}A
+\sum_{i=1}^2\int_{\partial\Sigma\cap P_{i}}\parens*{\inner{Y,R_{i}(p_{i}(\Phi(\bm{\nu})))}-\omega_{0}^i\inner{Y,\bm{m}_{i}}}\,\mathrm{d}s.
\]
Here, $H_1^F$ denotes the anisotropic mean curvature of $\Sigma$, $p_{i}$ is the projection onto the $\left\{\bm{\nu}, \bm{n}_{i}\right\}$-plane, $R_{i}$ is the $\pi / 2$-rotation in the $\left\{\bm{\nu}, \bm{n}_{i}\right\}$-plane and $\mathbf{m}_{i}$ is the conormal vector field of $\partial \Sigma \subset P_{i}$.
Furthermore, the map \( \Phi \) is defined as follows:
\[
    \Phi\colon\mathbb{S}^{n} \rightarrow \mathbb{R}^{n+1}, \quad \Phi(\xi) := \nabla F(\xi) + F(\xi) \xi.
\] 

By direct calculations, 
the first variational formula leads to
\[
H_1^F=const \, \text{ on } \Sigma \quad \text{ and } \quad \left\langle \Phi(\bm{\nu}),\bm{n}_{i}\right\rangle = \omega_{0}^{i} \quad \text { on } \partial \Sigma\cap P_{i}.
\]
Inspired by the discussion above, we can define an anisotropic $\bm{\omega}_0$-capillary hypersurface in the classical wedge $\mathbf{W}$ as one that satisfies the following condition on its boundary:
\[
\inner{\bm{\nu}^{F},\bm{n}_{i}} = \omega^{i}_{0} \quad \text {on} \quad \partial \Sigma\cap P_{i},
\]
where \(\bm{\omega}_0:=(\omega^{1}_0, \omega^{2}_0)\) is a constant vector
and \(\bm{\nu}^{F}:=\Phi(\bm{\nu})\) is the anisotropic normal vector field of \(\Sigma\).

Our goal in this paper is to present an Alexandrov-type theorem. 
Building upon the geometric proof by Montiel and Ros \cite{Montiel-Ros},
we aim to combine the Minkowski formulae and the Heintze-Karcher inequality. 
As is well known, the Minkowski formula for closed hypersurfaces in Euclidean space
can be derived by applying the divergence theorem to the tangential component of
the position vector field.
It is very interesting to observe that by applying the divergence theorem to the vector field below,
we can derive the  Minkowski formulae in a wedge, where 
\[
    X(x) = \inner{\bm{\nu}^F(x)-\bm{k}^{F},\bm{\nu}(x)}x
    -\inner{x,\bm{\nu}(x)}(\bm{\nu}^F(x)-\bm{k}^{F}).
\]
Here \(\bm{\nu}^F-\bm{k}^F\) is actually the anisotropic normal vector field
\(\bm{\nu}^{\bar{F}}\) with respect to a new
gauge \[\bar{F}(\xi):=F(\xi)-\inner{\xi,\bm{k}^{F}}.\]
Thus \(X(x)\) can be regarded as a vector triple  product \(\bm{\nu}\times (x\times \bm{\nu}^{\bar F})\)
in the \(3\)-dimensional space spanned by \(x,\bm{\nu}\) and \(\bm{\nu}^{\bar F}\). This observation allows us to avoid using the structure lemma in \cite[Proposition 2.4]{AS}, \cite[(15)]{LX} and \cite[Lemma 3.1]{JWXZ2}, thereby simplifying the calculations in the proof.
By performing parallel translations along \(\bm{\nu}^F-\bm{k}^F\), we further obtain the higher order Minkowski formulae.

\begin{theorem}\label{mink}
    Let \(\mathbf{W}\subset\R^{n+1}\) be a classical wedge
    and \(\Sigma\subset\overline{\mathbf{W}}\) be an immersed hypersurface.
    For \(i=1,2\), assume that
    \[
        \inner{\bm{\nu}^{F},\bm{n}_i}=\omega_{0}^i \quad \text{on} \quad \partial_i\Sigma,
    \]
    where \(\omega_{0}^i\) are constants.
    Suppose \(\bm{k}^{F}\) is a constant vector in \(\R^{n+1}\) satisfying
    \[
        \inner{\bm{k}^{F},\bm{n}_i} = \omega_{0}^i, \quad
        \text{for} \quad i=1,2.
    \]
    Then, for \(r\in \{1,\dots,n\}\), we have 
    \begin{equation}\label{highmink}
         \int_\Sigma  H^{F}_{r-1}(F(\bm{\nu})-\inner{\bm{k}^{F},\bm{\nu}})\,\mathrm{d}A
         = \int_\Sigma H_{r}^F\inner{x,\bm{\nu}}\,\mathrm{d}A.
    \end{equation}
    In particular, we have
    \begin{equation}\label{eq:Minkowski}
        \int_{\Sigma} (F(\bm{\nu})-\inner{\bm{k}^F, \bm{\nu}})\,\mathrm{d}A
        =\int_{\Sigma} H_1^F\inner{x, \bm{\nu}}\,\mathrm{d}A.
    \end{equation}
    When \(\omega_0^1=\omega_0^2=0\), i.e., \(\Sigma\) has free boundary, we can take \(\bm{k}^F\) to be the zero vector.
\end{theorem}

\begin{remark}
    The set of all vectors \(\bm{k}^F\) satisfying the conditions in Theorem \ref{mink} forms 
    an \((n-1)\)-dimensional affine subspace of \(\R^{n+1}\),
    whose normal space is \(\mathrm{span}\{\bm{n}_1,\bm{n}_2\}\).
\end{remark}

By carefully investigating the relationships of normal vector fields on each part of \(\partial\Sigma\),
we first prove that the anisotropic normal translation map \(\zeta\)
starting from the hypersurface \(\Sigma\) can cover the domain enclosed by the hypersurface.
Upon accomplishing this key step, similar to the closed hypersurface case,
we can finally establish the following Heintze-Karcher inequality
for anisotropic hypersurfaces with free boundary.       

\begin{theorem}\label{freeHK}
 Let \(\mathbf{W}\subset\mathbb{R}^{n+1}\) be a classical wedge
    and \(\Sigma \subset \overline{\mathbf{W}}\) be a smooth, compact, embedded,
    strictly anisotropic mean convex \(\bm{\omega}\)-capillary hypersurface with \(\omega^{i}(x)\leq 0\) (as defined in section \ref{anisotropiccapillary}). 
    Let \(\Omega\) be the enclosed domain by \(\Sigma\) and \(\partial\mathbf{W}\). 
    Then
    \[
        \int_{\Sigma} \frac{F(\bm{\nu})}{H_1^F} \mathrm{~d} A \geq (n+1)|\Omega|.
    \]
    Moreover, equality holds if and only if \(\Sigma\) is an anisotropic free boundary truncated Wulff shape.
\end{theorem}

\begin{remark}
    Given a constant vector \(\bm{k}_0\) defined in  \cite{JWXZ},
    we set \(F(\xi):=|\xi|-\inner{\bm{k}_0,\xi}\),
    then the above theorem reduces to Theorem 1.5 in \cite{JWXZ} for \(|\bm{k}_{0}|\leq 1\).
\end{remark}

Inspired by the work in \cite{WX},
we reduce the capillary case to the free boundary case by choosing a new gauge.
Thus we obtain the following
Heintze-Karcher type inequality and the Alexandrov type theorem.

\begin{theorem}\label{capiHKineq}
   For \(i=1,2\), let \(\omega_{0}^i\) be constants, \(\mathbf{W}\subset\mathbb{R}^{n+1}\) be a classical wedge and \(K\) be the closed body enclosed by \(\mathcal{W}:=\Phi(S^n)\).
   Assume that there exists a constant vector \(\bm{k}^{F}\) satisfying
    \begin{equation}
        \label{condition:kf}
        \inner{\bm{k}^{F},\bm{n}_i}=\omega^i_{0} \, \text{ for } \, i=1, 2 \quad
        \text{and} \quad \bm{k}^F\in K. \quad 
    \end{equation}
    Let \(\Sigma\subset\overline{\mathbf{W}}\) be a smooth, compact, embedded,
    strictly anisotropic mean convex hypersurface with
    \(\inner{\bm{\nu}^{F}(x),\bm{n}_{i}}:=\omega^i(x)\leq\omega_0^i\)
    for \(x\in\partial\Sigma\cap P_i,\,i=1,2\).
    Let \(\Omega\) be the enclosed domain by \(\Sigma\) and \(\partial\mathbf{W}\). 
    Then
\begin{equation}\label{hkineq}
    \int_{\Sigma}\frac{F(\bm{\nu})-\inner*{\mathbf{\bm{\nu}},\bm{k}^{F}}}{H_1^F}\mathrm{~d}A
        \geq (n+1)|\Omega|.
    \end{equation}
    Moreover, the equality holds if and only if
    \(\Sigma\) is a \(\bm{\omega}_0\)-capillary truncated Wulff shape.
\end{theorem}

\begin{theorem}\label{capialex}
    Let \(\omega_{0}^i\), \(\mathbf{W}\), and \(K\) be as defined in Theorem \ref{capiHKineq}.
    Assume that there exists a constant vector \(\bm{k}^{F}\) satisfying
    \[
        \inner{\bm{k}^{F}, \bm{n}_i} = \omega_{0}^i \quad \text{for } i=1,2 \quad
        \text{and} \quad \bm{k}^F\in K.
    \]
    Let \(\Sigma \subset \overline{\mathbf{W}}\) be a smooth embedded compact anisotropic \(\bm{\omega}_{0}\)-capillary hypersurface. Then
    
    (i) If \(\bm{k}^F \in \mathring{K}\) and \(\Sigma\) has constant \(r\)-th anisotropic mean curvature for some \(r \in \{1, \ldots, n\}\), \(\Sigma\) is an \(\bm{\omega}_{0}\)-capillary truncated Wulff shape.
    
    (ii) If \(\bm{k}^F \in K\) and \(\Sigma\) has constant anisotropic mean curvature, \(\Sigma\) is an \(\bm{\omega}_{0}\)-capillary truncated Wulff shape.
\end{theorem}

\begin{remark}
    (i) The vector \(\bm{k}^F\) required by Theorems \ref{capiHKineq} and \ref{capialex} may not always exist.
    Lemma \ref{lem:subcase} addresses this issue by breaking down the condition \eqref{condition:kf} into two separate cases, thereby clarifying the implications of the theorem. This provides a framework in the proof of Theorem \ref{capialex}. 
    
    (ii) If \(\bm{k}^F\) does exist, then \(\omega_{0}^i\) must be within the range \([-F(-\bm{n}_i), F(\bm{n}_i)]\) for \(i=1,2\).
    
    (iii)  When \(\bm{k}^{F}\in\partial K\), \(\bm{\omega}_{0}\)-capillary truncated Wulff shapes may not exist. Thus, the inequality \eqref{hkineq} is strict and there are no embedded compact anisotropic \(\bm{\omega}_{0}\)-capillary hypersurface in \(\overline{\mathbf{W}}\).
\end{remark}

\begin{remark}
    For convenience, we prove the free boundary case first
    and then derive the capillary case from the free boundary case using an idea presented in \cite{WX}.
    However, the technique of our proof can directly handle the capillary case,
    as shown in the subsequent proof.
\end{remark}

\begin{corollary}\label{noexist}
    Let \(\omega_{0}^i\), \(\mathbf{W}\), and \(K\) be as defined in Theorem \ref{capiHKineq}.
    Assume that there exists a constant vector \(\bm{k}^{F}\) satisfying:
    \[
        \inner{\bm{k}^{F}, \bm{n}_i} = \omega_{0}^i \quad \text{for } i=1,2.
    \]
    Then

    (i) there exists no smooth, compact, embedded, $\bm{\omega}_0$-capillary hypersurface with constant \(r\)-th anisotropic mean curvature for some $r \in\{2, \cdots, n\}$, such that $\Sigma \cap L=\varnothing$ and \(\bm{k}^F \in \mathring{K}\).
    
    (ii) there exists no smooth, embedded compact \(\bm{\omega}_{0}\)-capillary hypersurface with constant anisotropic mean curvature such that $\Sigma \cap L=\varnothing$ and \(\bm{k}^F \in K\). 
\end{corollary} 

The paper is organized as follows.
We begin in section \ref{Pre} by describing the anisotropic and geometric frameworks that will be utilized throughout the paper.
In section \ref{Minkowskisec}, 
we prove the Minkowski-type formulae in a wedge.
Section \ref{HKineqsec} presents the Heintze-Karcher inequality in a wedge.
In section \ref{alexthmsec},
we start with Lemma \ref{lem:subcase} to examine the existence of the vector \(\bm{k}^F\).
Subsequently, we explore the positivity of the \(r\)-th anisotropic mean curvature.
Finally, we derive the Alexandrov-type theorem for the free boundary case within a wedge and extend this result to the capillary case.

\section{Preliminaries}\label{Pre}

\subsection{Gauges and Wulff shapes}

Let \(F\colon\mathbb{S}^{n}\to\mathbb{R}\) be a nonnegative \(C^2\) function on \(\mathbb{S}^{n}\),
such that
\begin{equation}
    \label{cond:convex}
    \left(\nabla^2F+F\sigma\right)>0,
\end{equation}
where \(\sigma\) is the canonical metric on \(\mathbb{S}^{n}\),
\(\nabla\) is the gradient on \(\mathbb{S}^{n}\)
and \(\nabla^2\) is the Hessian on \(\mathbb{S}^{n}\).
We call \(F\) a \textit{gauge} on \(S^n\).
Although we do not assume \(F\) to be positive, the convexity condition \eqref{cond:convex} ensures that there is at most one point \(\xi_0\in S^n\) such that \(F(\xi_0)=0\).
If such a point exists, we say that $F$ is \textit{degenerate} or \(F\) \textit{degenerates} at \(\xi_0\).

The \textit{Cahn-Hoffman map} associated to \(F\) is given by
\[
    \Phi\colon\mathbb{S}^{n}\to\mathbb{R}^{n+1}, \quad \Phi(\xi):=\nabla F(\xi)+F(\xi)\xi.
\]
The unit \textit{Wulff shape} with respect to \(F\) is defined to be
\[
    \mathcal{W}=\Phi(S^n).
\]
\(\mathcal{W}\) encloses a closed strictly convex body \(K\) and \(\Phi\colon S^n\to\mathcal{W}\) is a diffeomorphism.
The Wulff shape centered at \(x_0\) with radius \(r_0\) is defined to be
\[
    \mathcal{W}_{r_0}(x_0) = \{x_0+r_0x\mid x\in\mathcal{W}\}.
\]
Let \(\tilde{F}\) be the positive one-homogeneous extension of \(F\) to \(\mathbb{R}^{n+1}\).
Then \(\tilde{F}\) is a gauge on \(\R^{n+1}\), as stated in the introduction.
We have \(\Phi(x)=D\tilde{F}(x)\) for $x\in S^n$, where \(D\) is the Euclidean derivative.

We collect some facts about \(F\) and \(K\).

\begin{proposition}\label{anisocsineq}
    The following statements hold:
    \begin{enumerate}
    \item \(\tilde{F}\) is the support function of \(K\), i.e.
    \[
        \tilde{F}(x)=\sup\{\inner{x,y}\mid y\in K\}.
    \]
    \item \(F(x)=\inner*{\Phi(x),x}\), for $x\in S^n$. 
    \end{enumerate}
\end{proposition}

\begin{proof}
    Due to \eqref{cond:convex}, \(\tilde{F}\) is a positive one-homogeneous convex function, thus it is sublinear.
    According to  Theorem 1.7.1 in \cite{Schneider}, there exists a unique convex body \(\tilde{K}\) with the support function \(\tilde{F}\).
    Since $\tilde{F}$ is differentiable at every \(u\in \mathbb{R}^{n+1}\backslash\{0\}\), by Corollary 1.7.3 of \cite{Schneider}, the support set 
    \(F(\tilde{K},u):=\{x\in \mathbb{R}^{n+1}|\inner{x,u}=\tilde{F}(u)\}\cap \tilde{K}=\{x\}\)  contains exactly one point and \(x=D\tilde{F}(u)=\Phi(u)\).
    This implies that  \(\partial \tilde{K}=\Phi(S^n)=\mathcal{W}\). Consequently, \(K=\tilde{K}\) associated with the support function \(\tilde{F}\).
    Hence, statements (1) and (2) are confirmed.
\end{proof}

By Proposition \ref{anisocsineq}, \(\Phi\) can be geometrically interpreted as follows: 
For any point \(x\in S^n\), \(\Phi(x)\) corresponds to the unique point \(y\in\mathcal{W}\) where the unit outward normal vector at \(y\) is precisely \(x\). Equivalently, \(\Phi\) is the inverse of the Gauss map of \(\mathcal{W}\).
Thus, the tangent map \(\mathrm{d}\Phi\) is an endomorphism of the tangent space \(T_xS^n\).
Define \(A^F\coloneqq\mathrm{d}\Phi\).
Then
\[
    \inner{A^F(X),Y}=\nabla^2(X,Y)+F\inner{X,Y}.
\]

\subsection{Hypersurfaces in a wedge}\label{Geometric setting}

Let \(\Sigma\) be a smooth, compact, embedded hypersurface within \(\overline{\mathbf{W}}\), which may contain corners.
We require the corners to be of at most codimension \(2\).
In other words, \(\Sigma\) is locally diffeomorphic
to an open set in \(\R_{\geq 0}^2\times\R^{n-2}\).
The hypersurface \(\Sigma\) can be decomposed as
\[
    \Sigma = \Sigma_0\sqcup\Sigma_1\sqcup\Sigma_2,
\]
with each \(\Sigma_i\) being a smooth manifold of dimension \(n-i\).
In this decomposition, \(\Sigma_0\) denotes the interior of \(\Sigma\), while \(\Sigma_1 \sqcup \Sigma_2\) represents the boundary of \(\Sigma\).
We assume \(\Sigma_0\subset\mathbf{W}\), \(\Sigma_1\subset P_1\cup P_2\) and \(\Sigma_2\subset L\).
We always assume \(\Sigma\) is connected.

The boundary \(\partial\Sigma\) of \(\Sigma\) is a closed
topological submanifold in \(\partial\mathbf{W}\).
Since \(\partial \mathbf{W}\) is homeomorphic to \(\R^{n}\),
Jordan's theorem ensures that \(\partial\Sigma\) has a well-defined interior \(\Gamma\) within \(\partial \mathbf{W}\).
\(\Sigma\cup\Gamma\) is a closed topological submanifold in \(\R^{n+1}\).
Consequently, again by Jordan's theorem,
there is a well-defined interior \(\Omega\) enclosed by \(\Sigma\cup\Gamma\) in \(\R^{n+1}\).

Let \(x\) denote the position vector of \(\Sigma\), and let
\(\bm{\nu}\) be the outward-pointing unit normal vector field along \(\Sigma\).
For \(i=1,2\), define \(\partial_i\Sigma\)
as the intersection of  \(\partial \Sigma\)
 with the closure of \(P_i\), i.e.,
\(\partial_i\Sigma=\partial\Sigma\cap\overline{P_i}\).
As a result, \(\partial_i\Sigma\) is a smooth manifold with boundary \(\Sigma_0\). 

Consider \(\partial_i\Sigma\) as a submanifold of \(\Sigma\).
This inclusion determines a unit normal vector field along \(\partial_i\Sigma\)
that points outside \(\Sigma\).
We denote it by \(\bm{\mu}_i\). 

Consider \(\partial_i\Sigma\) as a submanifold of \(P_i\).
This inclusion determines a unit normal vector field along \(\partial_i\Sigma\)
that points inside \(\Gamma\).
We denote it by \(\bm{m}_{i}\). 

Consider \(\Sigma_2\) as a submanifold of \(\partial_i\Sigma\).
This inclusion determines a unit normal vector field on \(\Sigma_2\)
that points outside \(\partial_i\Sigma\).
We denote it by \(\bm{\tau}_i\). 

Consider \(\Sigma_2\) as a submanifold of \(L\).
This inclusion determines a unit normal vector field along \(\Sigma_2\)
that points inside \(\Gamma\cap L\).
We denote it by \(\bm{l}\).
   
Fix \(i=1\) or \(2\).
Then 
the sets \(\{\bm{\nu}, \bm{\mu}_i\}\) and \(\{\bm{n}_i, \bm{m}_{i}\}\) span  the same plane, which is
orthogonal to \(\partial_i\Sigma\) at every point \(x\).
Moreover, the vectors \(\bm{\nu},\bm{n}_1,\bm{n}_2,\bm{\tau}_1,\bm{\tau}_2\) and \(\bm{l}\) are all contained in  a \(3\)-dimensional linear subspace
that is orthogonal to \(\Sigma_2\).

We will always work in the above setting.
It is illustrated in Figure \ref{fig:fields}.
We have the following lemma of these vector fields.

\begin{figure}[h]
    \includegraphics*{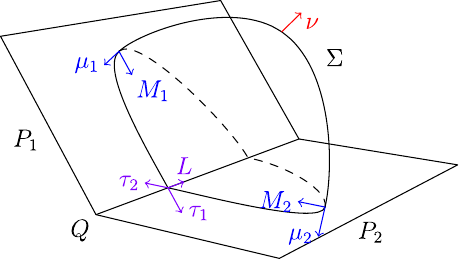}
    \caption{An illustration of the vector fields defined above.}
    \label{fig:fields}
\end{figure}

\begin{lemma}
    Let \(\Sigma\subset\overline{W}\) be a smooth compact embedded hypersurface in a wedge
    and \(\bm{\nu},\bm{n}_i,\bm{\mu}_i,\bm{m}_i,\bm{\tau}_i,\bm{l}\) be the vector fields defined above.
    Then
    \begin{enumerate}
        \item For \(i=1,2\), \(\bm{\mu}_i\) is a unit vector orthogonal to \(\partial_i\Sigma\) and
        \[
            \inner{\bm{\mu}_i,\bm{\nu}}=0, \quad
            \inner{\bm{\mu}_i,\bm{n}_i}\geq 0.
        \]
        \item For \(i=1,2\), \(\bm{m}_{i}\) is a unit vector orthogonal to \(\partial_i\Sigma\) and
        \[
            \inner{\bm{m}_{i},\bm{n}_i}=0, \quad
            \inner{\bm{m}_{i},\bm{\nu}}\leq 0.
        \]
        \item For \(i=1,2\), \(\bm{\tau}_i\) is a unit vector orthogonal to \(\Sigma_0\) and
        \[
            \inner{\bm{\tau}_i,\bm{\nu}}=\inner{\bm{\tau}_i,\bm{n}_i}=0, \quad
            \inner{\bm{\tau}_i,\bm{n}_{3-i}}\geq 0.
        \]
        \item \(\bm{l}\) is a unit vector orthogonal to \(\Sigma_0\) and
        \[
            \inner{\bm{l},\bm{n}_1}=\inner{\bm{l},\bm{n}_2}=0, \quad
            \inner{\bm{l},\bm{\nu}}\leq 0.
        \]
    \end{enumerate}
    If, in addition, \(\Sigma\) intersects with \(P_1,P_2\) and \(L\) transversally,
    i.e., \(\bm{\nu},\bm{n}_i\) are linearly independent along \(\partial_i\Sigma\)
    for \(i=1,2\) and \(\bm{\nu},\bm{n}_1,\bm{n}_2\) are linearly independent along \(L\),
    then \(\bm{\mu}_i,\bm{m}_i,\bm{\tau}_i,\bm{l}\) are uniquely determined by the above properties.
\end{lemma}

\begin{proof}
    The proof is quite trivial. 
    We will only prove (1) here.
    (2), (3) and (4) follow from almost the same arguments.

    For (1), \(\bm{\mu}_i\) is orthogonal to \(\partial_i\Sigma\)
    and tangent to \(\Sigma\) by definition.
    Since \(\bm{\nu}\) is orthogonal to \(\Sigma\), we have \(\inner{\bm{\mu}_i,\bm{\nu}}=0\).
    For each point \(x\in\partial_i\Sigma\),
    the fact that \(\bm{\mu}_i(x)\) points outside \(\Sigma\) implies that
    there is a smooth curve \(\alpha\) starting at \(x\) with initial velocity \(-\bm{\mu}_i(x)\)
    which lies entirely in \(\Sigma\).
    Since \(\Sigma\subset\overline{\mathbf{W}}\), \(\alpha\) lies entirely in \(\overline{\mathbf{W}}\).
    This implies that \(\inner{\bm{\mu}_i,\bm{n}_i}\geq 0\).
  
    It remains to show that there is exactly one vector field on \(\partial_i\Sigma\)
    satisfying these properties, provided that \(\Sigma\) intersects \(P_i\) transversally.
    In fact, if \(\Sigma\) and \(P_i\) are transversal, then \(\bm{\nu},\bm{n}_i\) are linearly independent.
    The normal space of \(\partial_i\Sigma\) is of dimension \(2\).
    Hence, \(\bm{\nu}, \bm{n}_ i\) form the basis of this normal space.
    As a result, \(\bm{\mu}_i\) is a unit vector spanned by \(\bm{\nu},\bm{n}_i\)
    and satisfies
    \[
        \inner{\bm{\mu}_i,\bm{\nu}}=0, \quad \inner{\bm{\mu}_i,\bm{n}_i}\geq 0.
    \]
    Such \(\bm{\mu}_i\) is obviously unique.
\end{proof}

\subsection{The anisotropic shape operator and anisotropic capillary hypersurfaces}\label{anisotropiccapillary}

Let \(\Sigma\) be a hypersurface in a wedge \(\overline{\mathbf{W}}\)
and let \(\bm{\nu}\) denote the unit outward normal vector field.
The \textit{anisotropic normal vector field} is defined as  \(\bm{\nu}^F=\Phi(\bm{\nu})\).
The \textit{anisotropic shape operator} is given by
\[
\mathrm{d}\bm{\nu}^F=\mathrm{d}\Phi|_{\bm{\nu}}\circ\mathrm{d}\bm{\nu} \colon T\Sigma\to T\Sigma.
\]
By linear algebra, the eigenvalues of \(\mathrm{d}\bm{\nu}^F\) are real, which are called the \textit{anisotropic principal curvatures} of \(\Sigma\) and denoted by \(\kappa_1^F,\ldots,\kappa_n^F\).
For \(r=0,1,\ldots,n\), the \(r\)-th \textit{anisotropic mean curvature} of \(\Sigma\) is defined as
\[
    H_r^F=\frac{1}{\binom{n}{r}}\sigma_r(\kappa_1^F,\ldots,\kappa_n^F),
\]
where \(\sigma_r\) represents the $r$-th elementary symmetric polynomial in \(n\) variables.

Let \(\omega^1,\omega^2\) be real functions on \(\Sigma\) and let \(\bm{\omega}=(\omega^1,\omega^2)\).
We call \(\Sigma\) an \textit{anisotropic \(\bm{\omega}\)-\text{capillary} hypersurface} in \(\overline{\mathbf{W}}\) if
\[
    \inner{\bm{\nu}^F,\bm{n}_i}=\omega^i \quad
    \text{on} \,\, \partial_i\Sigma
    \quad
    \text{for} \,\, i=1,2.
\]
If \(\omega^1=\omega^2=0\), we say \(\Sigma\) has \textit{anisotropic free boundary}.

Truncated Wulff shapes are always anisotropic capillary hypersurfaces in \(\overline{\mathbf{W}}\) with constant \(\bm{\omega}\).
In fact, suppose \(\mathcal{W}_r(y)\) is a Wulff shape
and intersects with \(\bar{P}_i\) at some point \(x\).
The anisotropic normal vector of \(\mathcal{W}_r(y)\) at \(x\) is
\[
    \bm{\nu}_{\scriptscriptstyle\mathcal{W}}^F=\frac{x-y}{r}.
\]
Thus, \(\inner{\bm{\nu}_{\scriptscriptstyle\mathcal{W}}^F,\bm{n}_i}=-\frac{1}{r}\inner{y,\bm{n}_i}\) is constant.
Write \(\omega_0^i=-\frac{1}{r}\inner{y,n_i}\) and \(\bm{\omega}_0=(\omega_0^1,\omega_0^2)\).
We say \(\mathcal{W}_r(y)\) is an \textit{\(\bm{\omega}_0\)-capillary Wulff shape} (even if it does not intersect with \(\partial{\mathbf{W}}\)).
If \(\mathcal{W}_r(y)\cap\overline{\mathbf{W}}\) is a hypersurface (with corners), we say \(\mathcal{W}_r(y)\cap\overline{\mathbf{W}}\) is a \textit{truncated \(\omega\)-capillary Wulff shape}.

Conversely, for any constant \(\bm{\omega}_0=(\omega_0^1,\omega_0^2)\),
let \(\bm{k}^F\in\R^{n+1}\) be a vector satisfying
\[
    \inner{\bm{k}^F,\bm{n}_i}=\omega_0^i \quad
    \text{for} \quad i=1,2.
\]
Then the Wulff shape \(\mathcal{W}_1(-\bm{k}^F)\) is an \(\bm{\omega}_0\)-capillary Wulff shape.

\section{Minkowski type formulae}\label{Minkowskisec}

In this section, we will initially establish a Minkowski-type formula for an anisotropic capillary hypersurface within a classical wedge, followed by the derivation of higher-order Minkowski type formulae.

\begin{proof}[Proof of Theorem \ref{mink}]
   Consider the vector field defined on \(\Sigma\) as follows
    \[
        X(x) = \inner{\bm{\nu}^F-\bm{k}^{F},\bm{\nu}}x-\inner{x,\bm{\nu}(x)}(\bm{\nu}^F(x)-\bm{k}^{F}).
    \]
   By applying the divergence theorem, we have
    \begin{equation}\label{eq:divX}
        \int_\Sigma\op{div}X = \int_{\partial_1\Sigma}\inner{X,\bm{\mu}_1}
        +\int_{\partial_2\Sigma}\inner{X,\bm{\mu}_2}.
    \end{equation}
    We claim that \(\inner{X,\bm{\mu}_i}=0\) for \(i=1,2\).

    In fact, \(X\) can be characterized by the property that
    \[
        \inner{X,Y} = \inner{(\bm{\nu}^F-\bm{k}^{F})\wedge x,\bm{\nu}\wedge Y}, \quad
        \text{for all} \quad Y\in T\Sigma.
    \]
    Then
    \[
        \inner{X,\bm{\mu}_i} = \inner{(\bm{\nu}^F-\bm{k}^{F})\wedge x,\bm{\nu}\wedge\bm{\mu}_i}.
    \]
    Along \(\partial_i\Sigma\), \(\bm{\nu}^F-\bm{k}^{F}\) and \(x\) are always orthogonal to \(\bm{n}_i\),
    while \(\{\bm{\nu}, \bm{\mu}_i\}\) and \(\{\bm{n}_i, \bm{m}_i\}\) span the same plane. 
    This implies that
    \[
        \inner{(\bm{\nu}^F-\bm{k}^{F})\wedge x,\bm{\nu}\wedge\bm{\mu}_i} = 0.
    \]
    Thus \eqref{eq:divX} reduces to
    \[
    \int_\Sigma\op{div}X = 0.
    \]
    We now show that \(\op{div}X = n(F(\bm{\nu}) - \inner{\bm{\nu}, \bm{k}^{F}} - H_1^F \inner{\bm{\nu}, x})\).
    
    Suppose \(e_1,\ldots,e_{n}\) is any local frame on \(\Sigma\).
    Then
    \[
        \op{div}X = g^{ij}\inner{\nabla_{e_i}X,e_j}.
    \]
    With the above characterization of \(X\), we have
    \begin{align*}
        \inner{\nabla_{e_i}X,e_j}
        & = \inner{D_{e_i}X,e_j} \\
        & = e_i\inner{X,e_j} - \inner{X,D_{e_i}e_j} \\
        & = \inner{(D_{e_i}\bm{\nu}^F)\wedge x,\bm{\nu}\wedge e_j}
        + \inner{(\bm{\nu}^F-\bm{k}^{F})\wedge e_i,\bm{\nu}\wedge e_j} \\
        & + \inner{(\bm{\nu}^F-\bm{k}^{F})\wedge x,(D_{e_i}\bm{\nu})\wedge e_j}.
    \end{align*}
    The third equality uses the fact that \(D_{e_i}\bm{k}^{F}=0\) and \(D_{e_i}x=e_i\).
    We compute each of the three terms separately. For the first term,
    \[
        g^{ij}\inner{(D_{e_i}\bm{\nu}^F)\wedge x,\bm{\nu}\wedge e_j}
        = -g^{ij}\inner{D_{e_i}\bm{\nu}^F,e_j}\inner{\bm{\nu},x}
        = -nH_1^F\inner{\bm{\nu},x},
    \]
  where  the first equality follows from \(\inner{D_{e_i}\bm{\nu}^F,\bm{\nu}}=0\). For the second term,
    \[
        g^{ij}\inner{(\bm{\nu}^F-\bm{k}^{F})\wedge e_i,\bm{\nu}\wedge e_j}
        = g^{ij}\inner{\bm{\nu}^F-\bm{k}^{F},\bm{\nu}}\inner{e_i,e_j}
        =n\inner{\bm{\nu}^F-\bm{k}^F, \bm{\nu}}
        = n(F(\bm{\nu})-\inner{\bm{k}^{F},\bm{\nu}}).
    \]
    where the first equality is derived from \(\inner{e_i,\bm{\nu}}=0\). For the third term,
    \[
        g^{ij}\,(D_{e_i}\bm{\nu})\wedge e_j
        = g^{ij}h_i^k\,e_k\wedge e_j
        = h^{jk}\,e_k\wedge e_j
        = 0,
    \]
   where the last equality is due to the symmetry of  \(h\).
    Thus, we have
    \[
        g^{ij}\inner{(\bm{\nu}^F-\bm{k}^{F})\wedge x,(D_{e_i}\bm{\nu})\wedge e_j} = 0.
    \]
 It follows that
    \[
        \op{div}X = n\left(F(\bm{\nu})
        -\inner{\bm{\nu},\bm{k}^{F}}-H_1^F\inner{\bm{\nu},x}\right).
    \]
    This completes the proof of \eqref{highmink} for $r=1$.
    
    Next, for small \(t\), we define
    \[
        \psi_t(x)=x+t(\bm{\nu}^{F}(x)-\bm{k}^{F}) \quad
        \text{for} \quad x\in\Sigma,
    \]
    which defines a family of parallel hypersurfaces $\Sigma_{t}$.

    On one hand, the \(\bm{\omega}\)-capillary condition
    and the definition of \(\bm{k}^{F}\) yield that
    for any \(x\in\partial\Sigma\cap P_{i}\),
    \[
        \inner{x+t(\bm{\nu}^F(x)-\bm{k}^F),\bm{n}_i}
        =t(\omega^{i}-\omega^{i})=0.
    \]
    This indicates that \(\psi_t(x)\in\partial\mathbf{W}\) for \(x\in\partial\Sigma\),
    which means \(\partial\Sigma_t\subset\partial\mathbf{W}\). 

    On the other hand, if \(e_1^{F},\cdots,e_n^F\)
    are anisotropic principal directions at \(x\in\Sigma\)
    corresponding to \(\kappa_i^F\) for \(i=1,\ldots,n\), we have
    \[
        (\psi_t)_*(e_i^F)=(1+t\kappa_i^F)e_i^F, \quad i=1,\ldots,n.
    \]

    These imply \(\bm{\nu}_{\Sigma_t}(\psi_t(x))=\bm{\nu}(x)\),
    so \(\bm{\nu}^F_{\Sigma_t}(\psi_t(x))=\bm{\nu}^F(x)\).
    Here \(\bm{\nu}_{\Sigma_t}\) and \(\bm{\nu}^F_{\Sigma_t}\) denote
    the outward normal vector field and anisotropic normal vector field to \(\Sigma_t\) respectively.
    Moreover, we have
    \[
        \inner{\bm{\nu}^F_{\Sigma_t}(\psi_t(x)),\bm{n}_{i}}
        = \inner{\bm{\nu}^F(x),\bm{n}_{i}}=\omega^{i}.
    \]

    Therefore, $\Sigma_t$ is also an anisotropic $\bm{\omega}$-capillary hypersurface in $\overline{\mathbf{W}}$ for any small $t$. So using \eqref{eq:Minkowski} for every such $\Sigma_t$, we find that
    \begin{equation}\label{when t}
        \int_{\Sigma_t=\psi_t(\Sigma)}
        \left( F(\bm{\nu}_t)-\inner{\bm{k}^{F},\bm{\nu}_{t}}
        - H_1^F(t)\inner{\psi_t,\bm{\nu}_t}\right) \mathrm{d} A_t=0.
    \end{equation}

    It is easy to see that the corresponding anisotropic principal curvatures are given by
    \[
        \kappa_i^F(\psi_t(x))=\frac{\kappa_i^F(x)}{1+t\kappa_i^F(x)} .
    \]
    Hence, if we denote \(\mathcal{P}_{n}(t)\) by
    \[
        \mathcal{P}_{n}(t)=\prod_{i=1}^{n}(1+t\kappa_{i}^{F})
        =\sum_{i=0}^{n}\binom{n}{i}H_{i}^{F}t^{i},
    \]
    then the anisotropic mean curvature of \(\Sigma_t\) at \(\psi_t(x)\) is given by
    \[
        H_1^F(t)=\frac{\mathcal{P}_{n}'(t)}{\mathcal{P}_{n}(t)}
        =\frac{\sum\limits_{i=0}^{n}i\displaystyle\binom{n}{i}H_i^Ft^{i-1}}{\mathcal{P}_{n}(t)}.
    \]

    Therefore, combining the tangential Jacobian of $\psi_t$ along $\Sigma$ at $x$
    \[
        \mathcal{J}^{\Sigma}\psi_t(x)=\prod_{i=1}^{n}(1+t \kappa_i^F(x))=\mathcal{P}_{n}(t).
    \]
    We can apply (\ref{when t}) to get
    \[
        \int_{\Sigma}(F(\bm{\nu})-\inner{\bm{k}^{F},\bm{\nu}})\mathcal{P}_{n}(t) 
        -\mathcal{P}_{n}^{\prime}(t)(\inner{x,\bm{\nu}}+t\inner{\bm{\nu}^{F},\bm{\nu}}
        -t\inner{\bm{k}^{F},\bm{\nu}})\,\mathrm{d} A_x=0 .
    \]

    Hence, by a direct computation, we obtain (\ref{highmink}).
\end{proof}

\begin{remark}
The same proof is also valid for general wedges bounded by multiple hyperplanes;
for a detailed definition, refer to \cite{JWXZ}.
\end{remark}

%%%%%%%%%%
\section{Heintze-Karcher type inequality}\label{HKineqsec}
%%%%%%%%%%
In this section, we aim to establish Theorem \ref{freeHK}.
First, we introduce the following lemma regarding monotonicity, which is proved in \cite[Proposition 3.1]{JWXZ2}.
We state it in a slightly different form.

\begin{lemma} \label{lem:monotoncity}
    Let \(z\) be a point in \(S^{n}\)
    and \(\gamma\colon[0,\pi]\to S^{n}\) be a unit speed geodesic such that \(\gamma(0)=z\).
    Then
    \[
        f(t) = \inner{\Phi(\gamma(t)),z}
    \]
    is a strictly decreasing function on \([0,\pi]\).
\end{lemma}

\begin{proof}
    We embed \(S^{n}\) into \(\R^{n+1}\) and let \(z\) also denote the position vector of itself.
    Then
    \begin{equation} \label{eq:monotoncity}
        f'(t) = \inner{D_{\gamma'(t)}\Phi(\gamma(t)),z}
        = \inner{A^F_{\gamma(t)}\gamma'(t),z}
        = \inner{A^F_{\gamma(t)}\gamma'(t),z^\perp},
    \end{equation}
    where \(z^\perp\) is the orthogonal projection of \(z\) onto \(T_{\gamma(t)}S^{n}\).
    The last equality follows from the fact
    that \(A^F_{\gamma(t)}\) is an endomorphism of \(T_{\gamma(t)}S^{n}\).
    Since \(z\) always lies in the plane spanned by \(\gamma(t)\) and \(\gamma'(t)\), we have
    \[
        z = \inner{z,\gamma(t)}\gamma(t)+\inner{z,\gamma'(t)}\gamma'(t).
    \]
    Hence, \(z^\perp=\inner{z,\gamma'(t)}\gamma'(t)\).
    Go back to \eqref{eq:monotoncity},
    \[
        f'(t) = \inner{z,\gamma'(t)}\inner{A^F_{\gamma(t)}\gamma'(t),\gamma'(t)}
        = \inner{z,\gamma'(t)}\left(\nabla^2F+F\sigma\right)(\gamma'(t),\gamma'(t)),
    \]
    where  \(\nabla^2F+F\sigma\) is positive.
    Notice  \(\gamma(t)\) can be written as
    \[
        \gamma(t) = \cos t\,z+\sin t\,\gamma'(0).
    \]
    Thus,
    \[
        \gamma'(t) = -\sin t\,z+\cos t\,\gamma'(0).
    \]
    Then
    \[
        \inner{z,\gamma'(t)} = -\sin t.
    \]
    This implies that \(f'(t)<0\) for \(t\in(0,\pi)\),
    thus proving the lemma.
\end{proof}

\begin{proof}[Proof of Theorem \ref{freeHK}]
    For any \(x\in\Sigma\), let \(\kappa_i^F(x)\) be the anisotropic principal curvatures
    and \(e_i^F(x)\) be the corresponding anisotropic principal directions.
    Since \(H_1^F>0\),
    \[
        \max_i\kappa_i^F(x) \geq H_1^F(x) > 0, \quad
        \text{for all} \quad x\in\Sigma.
    \]
    Define
    \[
        Z = \left\{(x,t)\in\Sigma\times\R\mid 0<t\leq\frac{1}{\max\limits_{i}\kappa_i^F(x)}\right\}
    \]
    and the anisotropic normal translation map
    \[
        \zeta\colon Z\to\R^{n+1}, \quad (x,t)\mapsto x-t\bm{\nu}^F(x).
    \]
    We claim that \(\Omega\subset\zeta(Z)\).
    
    For any \(y\in\Omega\), we consider a family of Wulff shapes \(\{\mathcal{W}_r(y)\}_{r>0}\).
    We want to show that \(\mathcal{W}_r(y)\cap\Sigma\neq\varnothing\) for some \(r>0\).
    If \(F\) is
    non-degenerate,
    \(\{\mathcal{W}_r(y)\}_{r>0}\) forms a foliation of \(\R^{n+1}\setminus\{y\}\).
    It is clear that such an \(r\) exists.
    If \(F\) is degenerate, it is a foliation of
    \[
        \mathcal{H}_y=\{x\in\R^{n+1}\mid\inner{x-y,\xi_0}<0\},
    \]
    where \(\xi_0\) is the unique unit vector such that \(F(\xi_0)=0\).
    The existence of such an \(r\) is
    due to the fact that 
    \(\mathcal{H}_y\cap\Sigma\neq\varnothing\) 
    for any \(y\in\Omega\).
        In fact, note that 
        \(\mathcal{H}_y\) is an unbounded connected open set
        and it is clear that \(\mathcal{H}_y\cap\Omega\neq\varnothing\).
        Since \(\Omega\) is bounded, we have \(\mathcal{H}_y\cap\partial\Omega\neq\varnothing\).
        \(\partial\Omega\) can be decomposed into a disjoint union \(\Sigma\sqcup\Gamma\).
        If \(\mathcal{H}_y\cap\Sigma\neq\varnothing\), then we are done.
        If \(\mathcal{H}_y\cap\Gamma\neq\varnothing\), consider the intersection \(A=\mathcal{H}_y\cap\partial\mathbf{W}\).
        \(A\) has \(1\) or \(2\) connected components and each component is unbounded.
        Since \(\Gamma\) is bounded, we have \(A\cap\partial\Sigma=A\cap\partial\Gamma\neq\varnothing\).
    
     Now take
    \[
        r_0 = \inf\{r>0\mid\mathcal{W}_r(y)\cap\Sigma\neq\varnothing\}.
    \]
    By the above arguments, \(r_0<+\infty\).
    Since \(\mathcal{W}_r(y)\subset\Omega\) when \(r\) is small, it follows that \(r_0>0\).
    Intuitively, this can be seen as a process where the Wulff shape continues to grow larger and larger
    until it comes into contact with \(\Sigma\).
    \(r_0\) represents  the moment when \(\mathcal{W}_r(y)\) touches \(\Sigma\) for the first time.
    Let \(D\) be the open region enclosed by \(\mathcal{W}_{r_0}(y)\).
    We have
    \begin{equation} \label{eq:inc}
        D\cap \mathbf{W} \subset \Omega.
    \end{equation}
    This inclusion is crucial in later discussions.

    Let \(x\) be a point in \(\mathcal{W}_{r_0}(y)\cap\Sigma\)
    and \(\bm{\nu}_{\scriptscriptstyle\mathcal{W}}(x)\) be the unit outward normal vector 
    of \(\mathcal{W}_{r_0}(y)\) at \(x\).
    We have
    \begin{equation} \label{eq:nuwf}
        \bm{\nu}_{\scriptscriptstyle\mathcal{W}}^F(x) = \frac{x-y}{r_0}.
    \end{equation}
    Next, we will always work in the tangent space \(T_x\R^{n+1}\).
    For simplicity, we will write the vector \(\bm{\nu}(x)\) as \(\bm{\nu}\).
    The same simplification applies to \(\bm{n}_i(x),\bm{\mu}_i(x),\bm{m}_{i}(x),\bm{\tau}_i(x)\) and \(\bm{l}(x)\),
    if they exist.
    There are three distinct cases.

    \textbf{Case 1.} \(x\in\Sigma_0\).
    In this case, \(\mathcal{W}_{r_0}(y)\) is tangent to \(\Sigma_0\),
    implying \(\bm{\nu}_{\scriptscriptstyle\mathcal{W}}=\bm{\nu}\).
    Therefore, \[\bm{\nu}^F=\bm{\nu}_{\scriptscriptstyle\mathcal{W}}^F=\frac{x-y}{r_0}.\]
    By the inclusion \eqref{eq:inc}, we have
    \[
        \kappa_i^F \leq \frac{1}{r_0}, \quad \text{for all}
        \quad i=1,\ldots,n.
    \]
    Hence, \((x,r_0)\in Z\) and \(y=\zeta(x,r_0)\in\zeta(Z)\).

    \textbf{Case 2.} \(x\in\Sigma_1\).
    Suppose \(x\in\partial_i\Sigma\), where \(i=1\) or \(2\).
    Then \(\mathcal{W}_{r_0}(y)\) is tangent to \(\partial_i\Sigma\).
    This implies that \(\bm{\nu}_{\scriptscriptstyle\mathcal{W}}\)
    lies in the plane spanned by \(\bm{\nu}\) and \(\bm{\mu}_i\).
    Hence, \(\bm{\nu},\bm{n}_i,\bm{\mu}_i,\bm{m}_{i},\bm{\nu}_{\scriptscriptstyle\mathcal{W}}\)
    all belong to the same plane.
    On one hand, by \eqref{eq:nuwf},
    \begin{equation} \label{eq:case2_ineq}
        \inner{\bm{\nu}_{\scriptscriptstyle\mathcal{W}}^F,\bm{n}_i} = -\frac{\inner{y,\bm{n}_i}}{r_0} > 0.
    \end{equation}
    On the other hand, by the inclusion \eqref{eq:inc}, we have
    \begin{equation} \label{eq:ang_ineq}
        \inner{\bm{\nu}_{\scriptscriptstyle\mathcal{W}},\bm{\mu}_i}\leq 0 \quad \text{and}
        \quad \inner{\bm{\nu}_{\scriptscriptstyle\mathcal{W}},\bm{m}_{i}}\leq 0.
    \end{equation}
    In fact, there is a smooth curve \(\alpha\)
    starting at \(x\) with initial velocity \(-\bm{\mu}_i\) which lies inside \(\Sigma\).
    By \eqref{eq:inc}, \(\alpha\) lies outside \(D\).
    Hence, \(\inner{\bm{\nu}_{\scriptscriptstyle\mathcal{W}},\bm{\mu}_i}\leq 0\).
    A similar argument gives \(\inner{\bm{\nu}_{\scriptscriptstyle\mathcal{W}},\bm{m}_{i}}\leq 0\).
    Surprisingly, if we represent \(\bm{\nu},\bm{n}_i,\bm{\nu}_{\scriptscriptstyle\mathcal{W}}\) as points on \(S^1\),
    \(\eqref{eq:ang_ineq}\) is equivalent to the statement
    that \(\bm{\nu}_{\scriptscriptstyle\mathcal{W}}\) lies on the length minimizing geodesic
    connecting \(\bm{\nu}\) and \(-\bm{n}_i\),
    or equivalently,
    that \(\bm{\nu}\) lies on the length minimizing geodesic
    connecting \(\bm{n}_i\) and \(\bm{\nu}_{\scriptscriptstyle\mathcal{W}}\).
    See Figure \ref{fig:case2} for an illustration.

    \begin{figure}[h]
        \includegraphics{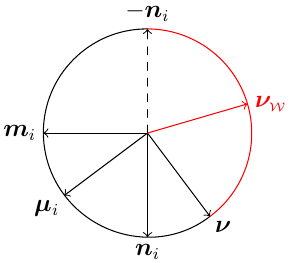}
        \caption{\(\bm{\nu}_{\scriptscriptstyle\mathcal{W}}\) must lies on the red arc.}
        \label{fig:case2}
    \end{figure}
        
    Therefore, by set $z=\bm{n}_{i}$ in Lemma \ref{lem:monotoncity}, we have
    \[
        \inner{\bm{\nu}_{\scriptscriptstyle\mathcal{W}}^F,\bm{n}_i} \leq \inner{\bm{\nu}^F,\bm{n}_i}=\omega^i \leq 0,
    \]
    which contradicts with \eqref{eq:case2_ineq}.
    This means that case 2 can not happen.

    \textbf{Case 3.} \(x\in\Sigma_2\).
    In this case, \(\mathcal{W}_{r_0}(y)\) is tangent to \(\Sigma_2\).
    Hence, \(\bm{\nu}_{\scriptscriptstyle\mathcal{W}}\) lies in the normal space of \(\Sigma_2\),
    which is a \(3-\)dimensional linear subspace containing \(\bm{\nu},\bm{n}_1,\bm{n}_2,\bm{\tau}_1,\bm{\tau}_2,\bm{l}\).
       
    On one hand, by \eqref{eq:nuwf}, we have
    \begin{equation} \label{eq:case3_ineq}
        \inner{\bm{\nu}_{\scriptscriptstyle\mathcal{W}}^F,\bm{n}_i} > 0, \quad
        \text{for} \quad i=1,2.
    \end{equation}
    On the other hand, by the inclusion \eqref{eq:inc}, we have
    \begin{equation} \label{eq:ang_ineq2}
        \inner{\bm{\nu}_{\scriptscriptstyle\mathcal{W}},\bm{\tau}_i}\leq 0 \quad
        \text{for} \quad i=1,2 \quad
        \text{and} \quad \inner{\bm{\nu}_{\scriptscriptstyle\mathcal{W}},\bm{l}}\leq 0.
    \end{equation}
    This follows from the same argument as in case 2.
    If we represent \(\bm{\nu},\bm{n}_1,\bm{n}_2,\bm{\nu}_{\scriptscriptstyle\mathcal{W}}\) as points in \(S^2\),
    \eqref{eq:ang_ineq2} is equivalent to the statement that
    \(\bm{\nu}_{\scriptscriptstyle\mathcal{W}}\) lies in
    the geodesic triangle with vertices \(\bm{\nu},-\bm{n}_1,-\bm{n}_2\)
    whose three sides are all length minimizing.
    See Figure \ref{fig:case3} for an illustration.

    \begin{figure}[h]
        \includegraphics{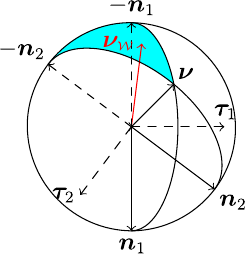}
        \caption{\(\bm{\nu}_{\scriptscriptstyle\mathcal{W}}\) must lies on the geodesic triangle with vertices \(-\bm{n}_1,-\bm{n}_2\) and \(\bm{\nu}\).
        The vector \(\bm{l}\) does not occur in this figure because it points perpendicularly into the paper.}
        \label{fig:case3}
    \end{figure}

    Denote this geodesic triangle by \(\Delta\).
    We also denote the length minimizing geodesic
    connecting \(\bm{n}_1\) and \(\bm{n}_2\) by \( \gamma \)
    and suppose \(\eta\) is a geodesic connecting
    \(\bm{\nu}_{\scriptscriptstyle\mathcal{W}}\) and \(\bm{\nu}\)
    which lies inside \(\Delta\).
    If we extend \(\eta\), it will intersect with \( \gamma \) at a certain point \(\bm{n}\).
    Moreover, \(\bm{n}\) can be written as
    \[
        \bm{n} = \lambda_1\bm{n}_1+\lambda_2\bm{n}_2 \quad
        \text{with} \quad \lambda_1,\lambda_2\geq 0.
    \]
    By \eqref{eq:case3_ineq}, we have
    \[
        \inner{\bm{\nu}_{\scriptscriptstyle\mathcal{W}}^F,\bm{n}} > 0.
    \]
    However, if we apply Lemma \ref{lem:monotoncity} with $z=\bm{n}$ to \(\eta\), we obtain
    \[
        \inner{\bm{\nu}_{\scriptscriptstyle\mathcal{W}}^F,\bm{n}}
        \leq \inner{\bm{\nu}^F,\bm{n}}
        = \lambda_1\omega_1+\lambda_2\omega_2\leq 0.
    \]
    This is a contradiction.
    Therefore, case 3 can not happen either.
    The claim is thus proved.

    Next, by direct computations, 
    the Jacobian of \(\zeta\) along \(Z\) at \((x,t)\) is
    \[
        \mathcal{J}\zeta(x,t) = F(\bm{\nu})\prod_{i=1}^{n}(1-t\kappa_i^F).
    \]
    By \(\Omega\subset\zeta(Z)\), the area formula yields
    \begin{align*}
        |\Omega| \leq \int_Z\mathcal{J}\zeta \,\mathrm{d}t\,\mathrm{d}A 
        & = \int_\Sigma\int_0^{\frac{1}{\max\limits_{i}{\kappa_i^F}}}
        F(\bm{\nu})\prod_{i=1}^{n}(1-t\kappa_i^F)\,\mathrm{d}t\,\mathrm{d}A \\
        & \leq \int_\Sigma\int_0^{\frac{1}{\max\limits_{i}{\kappa_i^F}}}F(\bm{\nu})(1-tH_1^F)^{n}\,
        \mathrm{d}t\,\mathrm{d}A \\
        & \leq \int_\Sigma\int_0^{\frac{1}{H_1^F}}F(\bm{\nu})(1-tH_1^F)^{n}\,\mathrm{d}t\,\mathrm{d}A \\
        & = \int_\Sigma\frac{F(\bm{\nu})}{(n+1)H_1^F}\,\mathrm{d}A.
    \end{align*}
        
    The second inequality comes from the AM-GM inequality.
    The third inequality follows from the fact that \(\max\limits_{i}{\kappa_i^F}\geq H_1^F>0\).
    If the equality holds, then \(\kappa_1^F(x)=\cdots=\kappa_{n}^F(x)\)
    for all \(x\in\Sigma\).
    It follows from  \cite{HeLi-Minkowski, Palmer} that \(\Sigma\) must be a part of a Wulff shape \(\mathcal{W}_r(y)\).
    If \(\partial_i\Sigma\neq\varnothing\), let \(x\) be a point in \(\partial_i\Sigma\).
    Then
    \[
        \inner{\bm{\nu}^F(x),\bm{n}_i} = \inner{\frac{x-y}{r},\bm{n}_i} = -\frac{\inner{y,\bm{n}_i}}{r},
    \]
    which is a constant. Denote it by \(\omega^i_0\).
    By the assumption that \(\Sigma\) is \(\bm{\omega}\)-capillary and \(\omega^i\leq 0\),
    we have \(\omega^i_0\leq 0\).
    If \(\omega^i_0<0\), it is easy to see that the difference of \(\Omega\) and \(\zeta(Z)\)
    have a positive measure,
    which violates the equality in Heintze-Karcher inequality.
    Hence, \(\omega^i_0=0\), i.e. \(\Sigma\) is a truncated Wulff shape with a free boundary.
\end{proof}

\begin{remark}
    The exclusion of Case 2 in the aforementioned proof is inherently consistent with the reasoning employed in \cite{JWXZ2} for the half-space case.
\end{remark}

\begin{proof}[Proof of Theorem \ref{capiHKineq}]

Consider $$\bar{F}(\xi):=F(\xi)-\langle \xi, \bm{k}^{F} \rangle.$$
Then the associated Cahn-Hoffman map \(\bar{\Phi}\) satisfies
\(\bar{\Phi}(x)=\Phi(x)-\bm{k}^F\), for \(x\in S^n\) and the unit Wulff shape with respect to \(\bar{F}\) is \(
\mathcal{W}_1(-\bm{k}^F)\).
Thus
$$\bm{\nu}^{\bar{F}}=\bar{\Phi}(\nu)=\Phi(\nu)-\bm{k}^F=\bm{\nu}^{F}-\bm{k}^{F}.$$
Notice that $\bar{F}$ is still a gauge. 
In fact, by $\bm{k}^{F}\in K$, it is easy to verify that $\bar{F}\in C^{2}(\mathbb{S}^n)$ and $\bar{F}(\xi)\geq0$. 
Furthermore,
$$\bar{F}_{ij}+\bar{F}\sigma_{ij}=F_{ij}+F \sigma_{ij}.$$
Hence,
$$\nabla^2\bar{F}(x)+\bar{F}(x)\sigma=\nabla^2F(x)+F(x)\sigma>0.$$
The conditions $\langle \bm{\nu}^{F}, \bm{n}_{i} \rangle=\omega^{i}$ and $\langle \bm{k}^{F}, \bm{n}_{i} \rangle=\omega_{0}^{i}$
imply $$\langle \bm{\nu}^{\bar{F}}, \bm{n}_{i} \rangle=\langle \bm{\nu}^{F},  \bm{n}_{i} \rangle-\langle k ^{F},\bm{n}_{i} \rangle
=\omega^i-\omega_0^i\leq 0. $$
This means that, with respect to the gauge \(\bar{F}\),  $\Sigma$ is an embedded compact strictly anisotropic mean convex anisotropic $\bar{\bm{\omega}}$-capillary hypersurface with $\bar{\omega}^{i}=\omega^i-\omega^i_0 \leq 0$ in the wedge. 

Since \(\mathrm{d}\bar{\Phi}=\mathrm{d}\Phi\), we have $H^{\bar{F}}_1=H_1^{F}$.
From the inequality in Theorem \ref{freeHK}, we immediately obtain the inequality in Theorem \ref{capiHKineq}. As for the equalities in Theorem \ref{capiHKineq}, we have known that $\Sigma$ is an anisotropic free boundary truncated Wulff shape associated with $\bar{F}$.

Then, from $$\bar{\Phi}(\mathbb{S}^{n})+\bm{k}^{F}=\Phi(\mathbb{S}^{n}),$$ 
we conclude that $\Sigma$ is an anisotropic $\bm{\omega}_{0}$-capillary truncated Wulff shape associated with $F$.  This completes the proof of Theorem \ref{capiHKineq}.
\end{proof}

\section{Alexandrov type theorem}\label{alexthmsec}

In this section, we begin by using the results from section \ref{HKineqsec} to establish the Alexandrov theorem for the free boundary case. Subsequently, we introduce a valuable idea and combine it with the above results to directly derive the corresponding result for capillary hypersurfaces.

Theorems \ref{capiHKineq} and \ref{capialex} rely on a constant vector \(\bm{k}^F\) satisfying \eqref{condition:kf}.
Such \(\bm{k}^F\) does not exist in general.  
Before proceeding with the proofs of the Theorem 
\ref{capialex}, it is essential to discuss the existence of the constant vector \(\bm{k}^F\). 
   
\begin{lemma}\label{existkF}
    \label{lem:subcase}
    Suppose there exists a constant vector \(\bm{k}^F\) satisfying \eqref{condition:kf}. Then only one of the following two cases can occur.
    \begin{enumerate}
        \item There exists a constant vector \(\tilde{\bm{k}}^F\) (which may differ from \(\bm{k}^F\))  satisfying  \eqref{condition:kf} and \(\tilde{\bm{k}}^F\in\mathring{K}\).
            This case occurs if and only if the \(\bm{\omega}_0\)-capillary
            Wulff shape in \(\overline{\mathbf{W}}\) intersects with \(L\) transversally.
        \item \(\bm{k}^F\) is the unique constant vector satisfying \eqref{condition:kf} and \(\bm{k}^F\in\mathcal{W}\).
            This case occurs if and only if the \(\bm{\omega}_0\)-capillary
            Wulff shape in \(\overline{\mathbf{W}}\) is tangent to \(L\).
    \end{enumerate}
\end{lemma}

\begin{proof}
    Since there exists a constant vector \(\bm{k}^F\) satisfying \eqref{condition:kf},
    \(\mathcal{W}_1(-\bm{k}^F)\) is an \(\omega_0\)-capillary truncated Wulff shape and  \(\bm{0}\) lies in the closed region enclosed by \(\mathcal{W}_1(-\bm{k}^F)\).
    \(L\) is a
    linear subspace which contains \(\bm{0}\).
    Thus, \(\mathcal{W}_1(-\bm{k}^F)\cap L\neq\varnothing\).
    Then either \(\mathcal{W}_1(-\bm{k}^F)\) intersects \(L\) transversally or \(\mathcal{W}_1(-\bm{k}^F)\) is tangent to \(L\).

    Suppose \(\mathcal{W}_1(-\bm{k}^F)\) intersects with \(L\) transversally.
    Then there are at least two distinct points \(x_1,x_2\in\mathcal{W}_1(-\bm{k}^F)\cap L\).
    Let
    \[
        \tilde{\bm{k}}^F
        =\frac{1}{2}\parens*{(\bm{k}^F+x_1)+(\bm{k}^F+x_2)}
        =\bm{k}^F+\frac{1}{2}(x_1+x_2).
    \]
    It is straightforward to check that \(\inner{\tilde{\bm{k}}^F,\bm{n}_i}=\omega_0^i\).
    On the other hand, since \(x_1,x_2\in\mathcal{W}_1(-\bm{k}^F)\),
    we have \(\bm{k}^F+x_1,\bm{k}^F+x_2\in\mathcal{W}\).
    By the strict convexity of \(K\),
    \(\tilde{\bm{k}}^F\in\mathring{K}\).

    Suppose \(\mathcal{W}_1(-\bm{k}^F)\) is tangent to \(L\).
    Then \(\mathcal{W}\) is tangent to \(\bm{k}^F+L\).
    Note that
    \[
        \bm{k}^F+L=\{x\in\R^{n+1}\mid\inner{x,\bm{n}_i}=\omega_0^i\text{ for }i=1,2\}
    \]
    is an \((n-1)\)-dimensional affine subspace of \(\R^{n+1}\).
    The strict convexity of \(K\) implies that \(\bm{k}^F\) is the unique point in \(\mathcal{W}\cap(\bm{k}^F+L)\).
\end{proof}

For the first case, let \(\bar{F}(x)=F(x)-\inner{\tilde{\bm{k}}^F,x}\).
Then \(\bar{F}\) is a nondegenerate gauge.
For the second case, let \(\bar{F}(x)=F(x)-\inner{\bm{k}^F,x}\).
Then \(\bar{F}\) is a gauge that degenerates at \(\bm{k}=\Phi^{-1}(\bm{k}^F)\).
Moreover, \(\bm{k}\) is the outward unit normal vector of \(\mathcal{W}\) at \(\bm{k}^F\).
Since \(\mathcal{W}\) is tangent to \(\bm{k}^F+L\), \(\bm{k}\in\op{span}\{\bm{n}_1,\bm{n}_2\}\).
The Cahn-Hoffman map of \(\bar{F}\) is \(\bar{\Phi}(x)=\Phi(x)-\tilde{\bm{k}}^F\) or \(\Phi(x)-\bm{k}^F\) and \(\mathrm{d}\bar{\Phi}=\mathrm{d}\Phi\).
Suppose \(\Sigma\subset\overline{\mathbf{W}}\) is an anisotropic \(\omega_0\)-capillary hypersurface with respect to \(F\).
Then, with respect to \(\bar{F}\), \(\Sigma\) is a hypersurface with anisotropic free boundary, and \(H_1^{\bar{F}}=H_1^F\).

\begin{proposition}\label{ellipticpoint}
    Let \(\Sigma\subset \overline{\mathbf{W}}\) be a smooth
    compact embedded anisotropic free boundary hypersurface.
    Suppose that $F$ is nondegenerate.
    Then \(\Sigma\) has at least one point at which all the anisotropic principal curvatures are positive.
\end{proposition}

\begin{proof}
    Fix a point \(y\in\R^{n+1}\) such that \(-y\in\mathbf{W}\).
    Consider the family of Wulff shapes \(\mathcal{W}_r(y)\) centred at \(y\).
    It is important to note that for any point
    \(x\in\partial\Sigma\cap
\mathcal{W}_r(y)\subset\partial\mathbf{W}\),
    the following relationship holds:
    \begin{equation}\label{tangent}
        \inner{\bm{\nu}^F_{\scriptscriptstyle\mathcal{W}}(x),\bm{n}_{i}}
        = \inner*{\frac{x-y}{r},\bm{n}_{i}}
        < 0.
    \end{equation}

    Since \(\Sigma\) is compact and $F$ is nondegenerate,
    there exists a sufficiently large radius \(r\) such that \(\Sigma\)
    is entirely enclosed by the Wulff shape \(\mathcal{W}_r(y)\).
    Consequently, we can find the smallest radius \(r_0>0\) at which
    \(\mathcal{W}_{r_0}(y)\) makes initial contact with \(\Sigma\)
    from the exterior at a point \(x_0\in\Sigma\).
    Let \(D\) be the open region enclosed by \(\mathcal{W}_{r_0}(y)\).
    We have
    \begin{equation} \label{eq:exc}
        \Omega\subset D\cap\mathbf{W} .
    \end{equation}

    Now, there are three cases regarding the position of \(x_{0}\).
    The proof is similar to the analysis in the Heintze-Karcher inequality above,
    and we will briefly describe it here.

    \textbf{Case 1.} \(x_0\in\Sigma_0\).
    Then \(\Sigma\) and \(\mathcal{W}_{r_0}(y)\) are tangent at \(x_{0}\).

    \textbf{Case 2.} \(x_{0}\in\Sigma_1\).
    Suppose \(x_{0}\in\partial_i\Sigma\).
    Then \(\mathcal{W}_{r_0}(y)\) is tangent to \(\partial_i\Sigma\),
    implying that \(\bm{\nu}_{\scriptscriptstyle\mathcal{W}}\) lies in the plane spanned by \(\bm{\nu}\) and \(\bm{\mu}_i\).
    Hence, \(\bm{\nu},\bm{\mu}_i, \bm{n}_i,\bm{m}_{i},\bm{\nu}_{\scriptscriptstyle\mathcal{W}}\) all belong to the same plane.
       
    By the inclusion \eqref{eq:exc}, we have
    \begin{equation} \label{exeq:ang_ineq}
        \inner{\bm{\nu}_{\scriptscriptstyle\mathcal{W}},\bm{\mu}_i}\geq 0 \quad \text{and}
        \quad \inner{\bm{\nu}_{\scriptscriptstyle\mathcal{W}},\bm{m}_{i}}\leq 0.
    \end{equation}
    If we regard \(\bm{\nu},\bm{n}_i,\bm{\nu}_{\scriptscriptstyle\mathcal{W}}\) as points on \(S^1\),
    \(\eqref{exeq:ang_ineq}\) is equivalent to the statement
    that \(\bm{\nu}\) lies on the length minimizing geodesic
    connecting \(\bm{\nu}_{\scriptscriptstyle\mathcal{W}}\) and \(-\bm{n}_i\),
    or equivalently,
    that  \(\bm{\nu}_{\scriptscriptstyle\mathcal{W}}\) lies on the length minimizing geodesic
    connecting \(\bm{n}_i\) and \(\bm{\nu}\).
    See Figure \ref{fig:case2a} for an illustration.

    \begin{figure}[h]
        \includegraphics{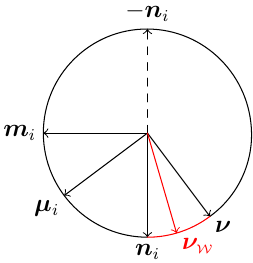}
        \caption{\(\bm{\nu}_{\scriptscriptstyle\mathcal{W}}\) must lies on the red arc.}
        \label{fig:case2a}
    \end{figure}

    Therefore, by setting \(z=\bm{n}_{i}\) in Lemma \ref{lem:monotoncity}, we have
    \[
        \inner{\bm{\nu}_{\scriptscriptstyle\mathcal{W}}^F,\bm{n}_i} \geq \inner{\bm{\nu}^F,\bm{n}_i}=0,
    \]
    Which contradicts with \eqref{tangent}.
    This means that case 2 can not happen.

    \textbf{Case 3.} \(x_{0}\in\Sigma_2\).
    In this case, \(\mathcal{W}_{r_0}(y)\) is tangent to \(\Sigma_2\).
    Hence, \(\bm{\nu}_{\scriptscriptstyle\mathcal{W}}\) lies in the normal space of \(\Sigma_2\),
    which is a \(3-\)dimensional linear subspace containing \(\bm{\nu},\bm{n}_1,\bm{n}_2,\bm{\tau}_1,\bm{\tau}_2,\bm{l}\).
       
    By the inclusion \eqref{eq:exc}, we have
    \begin{equation} \label{exeq:ang_ineq2}
        \inner{\bm{\nu}_{\scriptscriptstyle\mathcal{W}},\bm{\tau}_i}\geq 0 \quad
        \text{and} \quad \inner{\bm{\nu}_{\scriptscriptstyle\mathcal{W}},\bm{l}}\leq 0.
    \end{equation}
    If we regard \(\bm{\nu},\bm{n}_1,\bm{n}_2,\bm{\nu}_{\scriptscriptstyle\mathcal{W}}\) as points in \(S^2\),
    \eqref{exeq:ang_ineq2} is equivalent to the statement that
    \(\bm{\nu}_{\scriptscriptstyle\mathcal{W}}\) lies in
    the geodesic triangle with vertices \(\bm{\nu},\bm{n}_1,\bm{n}_2\)
    whose three sides are all length minimizing.

    Denote this geodesic triangle by \(\Delta\).
    Also denote the length minimizing geodesic connecting \(\bm{n}_1\) and \(\bm{n}_2\) by \( \gamma \).
    Suppose \(\eta\) is a geodesic connecting \(\bm{\nu}_{\scriptscriptstyle\mathcal{W}}\) and \(\bm{\nu}\)
    which lies inside \(\Delta\).
    If we extend \(\eta\), it will intersect with \( \gamma \) at some point \(\bm{n}\).
    See Figure \ref{fig:case3a} for an illustration.
    
    \begin{figure}[h]
        \includegraphics{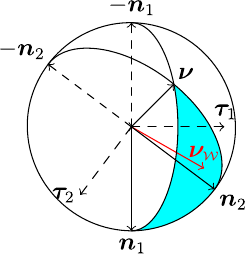}
        \caption{\(\bm{\nu}_{\scriptscriptstyle\mathcal{W}}\) must lies on the geodesic triangle with vertices \(\bm{n}_1,\bm{n}_2\) and \(\bm{\nu}\).
        }
        \label{fig:case3a}
    \end{figure}

    \(\bm{n}\) can be written as
    \[
        \bm{n} = \lambda_1\bm{n}_1+\lambda_2\bm{n}_2 \quad
        \text{with} \quad \lambda_1,\lambda_2\geq 0.
    \]
    By \eqref{tangent}, we have
    \[
        \inner{\bm{\nu}_{\scriptscriptstyle\mathcal{W}}^F,\bm{n}} < 0.
    \]
    However, if we apply Lemma \ref{lem:monotoncity} with $z=\bm{n}$ to \(\eta\), we obtain
    \[
       \inner{\bm{\nu}_{\scriptscriptstyle\mathcal{W}}^F,\bm{n}} \geq \inner{\bm{\nu}^F,\bm{n}} = 0.
    \]
    This is a contradiction.
    Hence, case 3 can not happen either.

    By the above discussion, \(\Sigma\) and \(\mathcal{W}_{r_0}(y)\) are tangent at a point \(x_0\in\Sigma_0\).
    Then the inclusion \eqref{eq:exc} implies that the anisotropic principal curvatures of \(\Sigma\)
    at \(x_0\) are larger than or equal to \(\frac{1}{r_0}\).
\end{proof}

   \begin{proposition}\label{HFpositive}
        Let \(\Sigma\subset \overline{\mathbf{W}}\) be a smooth
        compact embedded anisotropic free boundary hypersurface with constant anisotropic mean curvature $H_1^F$.
        If there exists $\xi\in S^n\cap \mathrm{span}\{\bm{n}_1, \bm{n}_2\}$ such that $F(\xi)=0$, then 
       $H_1^F> 0$.
    \end{proposition} 

    \begin{proof}
        Let \(\bm{l}_1,\bm{l}_2\) be vectors in the normal space \(N_{\bm{0}}L\) of $L$ such that
        \[
            \inner{\bm{l}_i,\bm{n}_i}=0 \quad \text{and} \quad \inner{\bm{l}_i,\bm{n}_{3-i}}<0 \quad
            \text{for \(i=1,2.\)}
        \]
        Now we consider a family of half-spaces
        \[
            \mathcal{H}_a=\{x\in\R^{n+1}\mid\inner{x,\xi}<a\}, \quad a\in\R
        \]
        
        Case 1: \(\inner{\xi,\bm{l}_i}\leq0\) for \(i=1,2\).
        In this case, \(\overline{\mathbf{W}}\subset\overline{\mathcal{H}}_0\).
        We can use the same method as Proposition \ref{ellipticpoint} to get the result.

        Case 2: \(\inner{\xi,\bm{l}_i}>0\) for \(i=1\) or \(2\).
        By the compactness of $\overline{\Omega}$, for \(a\) large enough, \(\Omega\subset\mathcal{H}_a\).
        Let
        \[
            a_0=\sup\{a\in\R\mid\partial\mathcal{H}_a\cap\Sigma\neq\varnothing\}.
        \]
        Then we have
        \[
            \partial\mathcal{H}_{a_0}\cap\Sigma\neq\varnothing \quad
            \text{and} \quad \Omega\subset\mathcal{H}_{a_0}.
        \]

        Let \(x_0\) be a point of \(\partial\mathcal{H}_{a_0}\cap\Sigma\).
        If \(x_0\in\mathring{\Sigma}\), then \(\partial\mathcal{H}_{a_0}\)
        is tangent to \(\Sigma\) at \(x_0\).
        If \(x_0\in\partial\Sigma\), assume \(x_0\in\partial_i\Sigma\).
        Then, at \(x_0\), we have
        \[
            \inner{\xi,\bm{m}_i}\leq 0, \quad \inner{\xi,\bm{\mu}_i}\geq 0.
        \]
        These two inequalities are obtained from the inclusion \(\Omega\subset\mathcal{H}_{a_0}\).
        Note that
        \[
            \inner{\Phi(\xi),\bm{n}_i}=
0=\inner{\bm{\nu}^F,\bm{n}_i}.
        \]
        By Lemma \ref{lem:monotoncity}, \(\xi=\bm{\nu}\) which means that
        \(\partial\mathcal{H}_{a_0}\) is still tangent to \(\Sigma\) at \(x_0\).

        Recall that if a hypersurface $\Sigma\in\mathbb{R}^{n+1}$ is given by a smooth function  $u$ defined on a domain of $\mathbb{R}^n$,  with respect to the upper unit normal vector field $\nu=\frac{(-\nabla u, 1)}{\sqrt{1+|\nabla u|^2}}$, the anisotropic mean curvature of $\Sigma$ is given by  (See for example \cite{W06})
\begin{equation}\label{eq:H^F-graph}
\sum_{i=1}^n \frac{\partial}{\partial x_i}(\frac{\partial f}{\partial p_i}(\nabla u))=H_1^F,
\end{equation}
where $f(p)=F(-p,1)$ for any $p\in \mathbb{R}^n$. 
From the assumption \eqref{cond:convex} on $F$, we see that  $H_1^F=const.$ is a quasilinear elliptic equation. 

We claim that $H_1^F > 0$. Otherwise, suppose that $H_1^F\leq 0$. Let $u_1,u_2$ be the graph functions of the hyperplane $\partial \mathcal{H}_{a_0}$ and $\Sigma$ in a neighborhood $U$ of $x_0$, respectively, and $\Sigma $ touches $\partial \mathcal{H}_{a_0}$ at $x_0$. Let $w:=u_1-u_2$. Then $w$ attains its nonnegative maximum $0$ at $x_0$ and $Lw\geq 0$ for a uniformly elliptic operator $L$.

If $x_0$ is an interior point of $\Sigma$, by the strong maximum principle, $w$ is constant.
Additionally, since $w(x_0)=0$, it follows that $u_2=u_1$.
This implies that $\Sigma$ and $\partial \mathcal{H}_{a_0}$ agree in $U$.
When \( H_1^F = \text{constant} \), this method can be  applied at points on \( \partial U \cap \mathring{\Sigma} \) to deduce that \( \Sigma \) and \( \partial \mathcal{H}_{a_0} \) agree in a larger domain $\tilde{U}$ such that $U\subset\subset \tilde{U}$.
By continuously repeating this process, we can ensure that \( \Sigma \) and \( \partial \mathcal{H}_{a_0} \) are in agreement within \( \overline{\mathbf{W}} \).
This results in a contradiction to the assumptions on \(\Sigma\) as outlined in Subsection \ref{Geometric setting}.

If $x_0\in \partial \Sigma$, according to the Hopf lemma,  $w$ is either constant or $\frac{\partial w}{\partial \nu}(x_0)>0$, where $\nu$ denotes the unit outward normal vector field along the boundary. The second case is incompatible with the condition \(\frac{\partial w}{\partial \nu}(x_0) = 0\), since by the above discussion, $\partial \mathcal{H}_{a_0}$ is tangent to $\Sigma$ at $x_0$. Thus, applying the reasoning for the interior point situation above leads to a contradiction.
\end{proof}

Combining the above two propositions, we get the following corollary.

\begin{corollary}\label{hmcpositve}
    Let \(\Sigma\subset \overline{\mathbf{W}}\) be a smooth
    compact embedded anisotropic \(\bm{\omega}_0\)-capillary hypersurface.
    Assume there is a constant vector \(\bm{k}^F\) satisfying \eqref{condition:kf}.
    In the first case of Lemma \ref{lem:subcase},
    there exists a point \(x\in\Sigma\) at which the anisotropic principal curvatures are positive.
    In the second case of Lemma \ref{lem:subcase}, if the anisotropic mean curvature $H_1^F$ is constant, then \(H_1^F>0\).
\end{corollary}

Now we get the Alexandrov-type theorem for the free boundary case as follows.

\begin{proposition}\label{freebdythm}
        Let \(\Sigma\subset\overline{\mathbf{W}}\) be a smooth embedded compact
    anisotropic free boundary hypersurface.  Assume there is a constant vector \(\bm{k}^F\) satisfying \eqref{condition:kf}.
  
  (i) In the first case of Lemma \ref{lem:subcase}, if the anisotropic
    \(r\)-th mean curvature \(H_r^F\) is constant for some \(r\in\{1,..., n\}\),
    then \(\Sigma\) is an anisotropic free boundary truncated Wulff shape.

 (ii)  In the second case of Lemma \ref{lem:subcase}, 
    if the anisotropic
    mean curvature \(H_1^F\) is constant,
    then \(\Sigma\) is an anisotropic free boundary truncated Wulff shape.
\end{proposition}

\begin{proof}
In the first case of Lemma \ref{lem:subcase}, according to Proposition \ref{ellipticpoint}, there exists a point \(x\in\Sigma\)  where the principal curvatures \(\kappa_1^F,\ldots,\kappa_n^F\) are all positive.
Thus, \(H_r^F\) is a positive constant. 
It follows from G\r{a}rding's result \cite{Montiel-Ros} that \(H_k^F>0\) for \(1\leq k\leq r\).
Applying Theorem \ref{freeHK} and using the Maclaurin inequality $H_1^F \geq\left(H_r^F\right)^{1 / r}$ with $H_r^F$ being constant, we have
\begin{equation}\label{eq:crmc}
(n+1)\left(H_r^F\right)^{1 / r}|\Omega| \leq\left(H_r^F\right)^{1 / r} \int_{\Sigma} \frac{F(\nu)}{H_1^F}\, \mathrm{d} A  \leq \int_{\Sigma}F(\nu)\, \mathrm{d} A .
\end{equation}
On the other hand, employing Theorem \ref{mink} and the Maclaurin inequality $H_{r-1}^F \geq\left(H_r^F\right)^{\frac{r-1}{r}}$, we obtain
$$
\begin{aligned}
0 & =\int_{\Sigma} \left(H_{r-1}^F F(\nu)-H_r^F\langle x, \nu\rangle \right)\, \mathrm{d} A 
\geq  \int_{\Sigma}\left(\left(H_r^F\right)^{\frac{r-1}{r}} F(\nu)-H_r^F\langle x, \nu\rangle \right)\mathrm{d} A \\
& =\left(H_r^F\right)^{\frac{r-1}{r}} \int_{\Sigma} F(\nu) \mathrm{d} A - (n+1)H_r^F |\Omega|.\\
\end{aligned}
$$
Thus, equality in \eqref{eq:crmc}   
holds, and by  Theorem \ref{freeHK}, $\Sigma$ is an anisotropic free boundary truncated Wulff shape.
The proof for the second case of Lemma \ref{lem:subcase} follows the same steps as the first case,  with \(r\) set to \(1\).
\end{proof}

\begin{proof}[Proof of Theorem \ref{capialex}]
Similar to the proof of Theorem \ref{capiHKineq}, we consider $$\bar{F}(\xi):=F(\xi)-\langle \xi, \bm{k}^{F} \rangle.$$
Then 

(i) $\bm{\nu}^{\bar{F}}=\bm{\nu}^{F}-\bm{k}^{F};$

(ii) $\bar{F}$ is a gauge with $\nabla^2\bar{F}(x)+\bar{F}(x)\sigma>0 \text{ for } x\in \mathbb{S}^{n};$

(iii) $H_r^{\bar{F}}=H_r^{F}.$

Hence, $\langle \bm{\nu}^{F},  \bm{n}_{i} \rangle=\omega_0^{i} \text{ on }\partial \Sigma$ and $\langle \bm{k}^{F} , \bm{n}_{i} \rangle=\omega_{0}^{i}$
implies that the free boundary condition $$\langle \bm{\nu}^{\bar{F}}, \bm{n}_{i} \rangle=\langle \bm{\nu}^{F},  \bm{n}_{i} \rangle-\langle k ^{F},\bm{n}_{i} \rangle
= 0$$
holds. This indicates that, with respect to the gauge \(\bar{F}\), \(\Sigma\) is an embedded compact anisotropic free boundary hypersurface within the wedge.
Additionally, in the first case outlined in Lemma \ref{lem:subcase}, the constant \(H^F_r\) is positive, and in the second case, the constant \(H^F_1\) is positive.

Thus, from Proposition $\ref{freebdythm}$ and $$\bar{\Phi}(\mathbb{S}^{n})+\bm{k}^{F}=\Phi(\mathbb{S}^{n}),$$ 
we obtain $\Sigma$ is an anisotropic $\bm{\omega}_{0}$-capillary truncated Wulff shape associated with $F$.  
\end{proof}  

\begin{remark}
    The above technique can be used to prove the Heintze-Karcher type inequality and Alexandrov-type theorem for anisotropic capillary hypersurfaces in some types of circular cones.
\end{remark}

\begin{proof}[Proof of Corollary \ref{noexist}]

By Corollary \ref{hmcpositve}, in the first case of Lemma \ref{lem:subcase}, it is established that \( H^F_r \) is a positive constant, and in the second case of Lemma \ref{lem:subcase}, \( H^F_1 \) is a positive constant.

Following the reasoning in the proof of Proposition  \ref{freebdythm}, we can combine Theorem \ref{mink} with Theorem \ref{capiHKineq} to deduce that the equality condition in the Heintze-Karcher inequality \eqref{hkineq} is satisfied. Consequently, \( \Sigma \) must be an \( \omega_0 \)-truncated Wulff shape. However, this conclusion conflicts with the condition \( \Sigma \cap L = \varnothing \) as stated in Lemma \ref{existkF}.
\end{proof}

%%%%%%%%%%%%%%%%
\begin{ack}
The authors would like to thank Professors Chao Qian, Chao Xia, Qiaoling Xia, Zhenxiao Xie and Ge Xiong for their helpful comments.
The authors were partially supported by the 
National Natural Science Foundation of China (grant numbers 12471048, 11831005 and 12061131014).
\end{ack}

\end{document}